\documentclass[12pt]{amsart}

\usepackage{amsmath,amsthm,amsfonts,amscd,flafter,epsf}

\usepackage{amssymb} % in order to get the semidirect product symbol

\usepackage{pdfsync}
 
\title[{Coarse Structures on Groups}]
{Coarse Structures on Groups} %running header

 \author[Nicas]{Andrew Nicas$^*$}
\address{Department of Mathematics and Statistics, McMaster University, Hamilton, Ontario, Canada L8S 4K1}
\email{nicas@mcmaster.ca}
\thanks{$^*$Partially supported by a grant from
the Natural Sciences and Engineering Research Council of Canada}
 
\author[Rosenthal]{David Rosenthal$^\dagger$}
\address{Department of Mathematics and Computer Science, St. JohnÕs University, 8000 Utopia
Pkwy, Jamaica, NY 11439, USA}
\thanks{$^\dagger$Partially supported by a Fulbright Scholars award.}
\email{rosenthd@stjohns.edu}

\date{January 14, 2012}
\subjclass[2010]{Primary  20F69; Secondary 54H11}
\keywords{coarse structure, asymptotic dimension, topological group}

\numberwithin{equation}{section}
    
\begin{document}

\begin{abstract} 
We introduce the {\it group-compact coarse structure} on a Hausdorff topological group in the context of coarse structures on an abstract group which are compatible with the group operations.   We develop asymptotic dimension theory for the group-compact coarse structure generalizing several familiar results for discrete groups.  We show that the asymptotic dimension in our sense  of the free topological group on a non-empty topological space that is homeomorphic to a closed subspace of a Cartesian product of metrizable spaces is $1$.
\end{abstract}

\maketitle

\baselineskip 18pt

%%%%%%%%%%%%%%%%%%%%%%%%%%%%%%%%%%%%%%%%%%%%
%%
%%  LaTeX environments and definitions 
%%
%%%%%%%%%%%%%%%%%%%%%%%%%%%%%%%%%%%%%%%%%%%%

%% Environments
\newtheorem{theorem}{Theorem}[section]
\newtheorem{lemma}[theorem]{Lemma}
\newtheorem{proposition}[theorem]{Proposition}
\newtheorem{corollary}[theorem]{Corollary}
\newtheorem{claim}[theorem]{Claim}
\newtheorem*{theorem*}{Theorem}

\theoremstyle{definition}
\newtheorem{definition}[theorem]{Definition}

\theoremstyle{remark}
\newtheorem{remark}[theorem]{Remark}

\theoremstyle{definition}
\newtheorem{example}[theorem]{Example}

\newcommand{\ra}{{\rightarrow}}
\newcommand{\bs}{{\backslash}}
\newcommand{\lra}{{\longrightarrow}}

% operator names

\newcommand{\asdim}{\operatorname{asdim}}
\newcommand{\id}{\operatorname{id}}
\newcommand{\supp}{\operatorname{supp}}
\newcommand{\Ebar}{\underbar{\rm{E}}}

\newcommand{\tb}{{\large \textbullet}}

\newcommand{\Zhalf}{{\mathbb Z}[1/2]}

%\newcommand{\del}{{\partial}}

% calligraphic letters
\newcommand{\cC}{\mathcal{C}}
\newcommand{\cD}{\mathcal{D}}
\newcommand{\cE}{{\mathcal E}}
\newcommand{\cF}{{\mathcal F}}
\newcommand{\cH}{{\mathcal H}}
\newcommand{\cJ}{{\mathcal J}}
\newcommand{\calo}{\mathcal{O}}
\newcommand{\cP}{{\mathcal P}}
\newcommand{\cT}{{\mathcal T}}
\newcommand{\cU}{{\mathcal U}}
\newcommand{\cV}{{\mathcal V}}
\newcommand{\cW}{{\mathcal W}}

% fractur letters
\newcommand{\N}{\mathfrak{N}}

% hatted characters
%\newcommand{\hPhi}{{\widehat \Phi}}

% tilde characters
\newcommand{\tS}{{\widetilde S}}

% barred characters
%\newcommand{\bPhi}{{\bar \Phi}}

%roman characters
\newcommand{\B}{{\rm B}}
\newcommand{\E}{{\rm E}}

% bold characters
\newcommand{\K}{{\mathbf K}}
\newcommand{\HH}{{\mathbf{HH}}}

% black board bold characters
\newcommand{\bN}{{\mathbb N}}
\newcommand{\R}{{\mathbb R}}
\newcommand{\Z}{{\mathbb Z}}

% underlined characters
%\newcommand{\n}{{\underline n}}

%%%%%%%%%%%%%%%%%%%%%%%%%%%%%%%%%%%%%%%%%%%%
%%
%%  Section Introduction
%%
%%%%%%%%%%%%%%%%%%%%%%%%%%%%%%%%%%%%%%%%%%%%

\section{Introduction}   %delete the * to get a numbered introduction

The notion of {\it asymptotic dimension} was introduced by Gromov as a tool for studying the large scale geometry of groups.
Yu stimulated widespread interest in this concept when he proved that the Baum-Connes assembly map in topological $K$-theory is a split injection for torsion-free groups with finite asymptotic dimension~\cite{Yu}. 
The asymptotic dimension of a metric space $(X,d)$ is defined to be the smallest integer $n$ such that for any positive number $R$, there exists a uniformly bounded cover of $X$ of multiplicity less than or equal to $n+1$ whose Lebesgue number is at least $R$ (if no such integer exists we say that the asymptotic dimension of $(X,d)$ is infinite).
A finitely generated group can be viewed as a metric space by giving it the word-length metric with respect to a given finite generating set.  The asymptotic dimension of this metric space is independent of the choice of the finite generating set and hence is an invariant of the group. The class of groups that have finite asymptotic dimension includes word hyperbolic groups, cocompact discrete subgroups of virtually connected Lie groups and mapping class groups.  However, there exist finitely generated groups, indeed finitely presented groups,  with infinite asymptotic dimension, for example Thompson's group $F$.

Roe generalized the notion of asymptotic dimension to {\it coarse spaces} \cite[\S 2]{Roe}.
A {\it coarse structure} on a set $X$ is a collection of subsets of $X \times X$ called {\it entourages} or {\it controlled sets} satisfying certain axioms (see Definition  \ref{def:coarse_structure}).  A set together with a coarse structure is a coarse space.
For a metric space $(X,d)$ equipped with the {\it bounded} or {\it metric coarse structure}  Roe's definition reduces to the original definition of asymptotic dimension for $(X,d)$.

We say that a coarse structure on an abstract group $G$ is {\it compatible} if every entourage is contained in a $G$-invariant entourage (Definition \ref{def:compatible_coarse}). 
We show that any such coarse structure on $G$ is obtained from a {\it generating family}, that is, a collection $\cF$ of subsets of $G$ satisfying certain axioms
(listed in Definition \ref{generating_family}), by means of the following construction.
Given a generating family $\cF$,  the collection
\[
\cE_{\cF} = \{ E \subset G \times G ~{\big |}~  \text{there exists } A \in \cF \text{ such that }  E \subset G(A \times A) \}
\]
is a compatible coarse structure on $G$.
For example, the collection $\cF=\operatorname{fin}(G) $ of all finite subsets of $G$ is a generating family (Example \ref{EX_finite_subsets})
and we call the corresponding coarse structure  the {\it group-finite coarse structure}.
If $G$ is a finitely generated group then the group-finite coarse structure coincides with the bounded coarse structure for a word metric on $G$;
indeed, this remains valid for a countable, infinitely generated group $G$  for an appropriate ``weighted'' word metric on $G$ corresponding to an infinite generating set (\cite[Remark 2]{Dranish_and_Smith}).

In this paper we  introduce the {\it group-compact coarse structure} on an arbitrary Hausdorff topological group $G$ (Example \ref{EX_group_compact}).
This coarse structure corresponds to the generating family $\cF = \cC(G)$  consisting of all compact subsets of $G$ and thus  depends only on the group structure and topology of $G$.
In particular, the asymptotic dimension of a Hausdorff topological group $G$, which we denote by $\asdim(G)$,
 is well-defined as the asymptotic dimension of $G$ with respect to the group-compact coarse structure.
When $G$ admits a left-invariant metric such that the bounded subsets with respect to the metric are precisely the relatively compact subsets with respect to the given topology of $G$, then the group-compact coarse structure coincides with the bounded coarse structure on $G$ (see Theorem~\ref{group-compact_equals_metric}).
However, not every $G$ admits such a metric (see Proposition \ref{no_left_invariant}).
Our definition of asymptotic dimension for a Hausdorff topological group $G$ is sensitive to the topology of $G$.
For example, if one considers the additive group of real numbers $\R$ with its usual topology, then $\asdim( \R )=1$, whereas if $\R$ is given the discrete topology, then its asymptotic dimension is infinite, since it contains closed subgroups isomorphic to $\Z^n$  for every $n$ and $\asdim( \Z^n)=n$.

Many of the facts about classical asymptotic dimension for finitely generated groups have analogs for our generalized definition of asymptotic dimension.
For example, if $G$ is a Hausdorff topological group with a compact set of generators, then the asymptotic dimension of $G$ with respect to the group-compact coarse structure is zero if and only if $G$ is compact~(Corollary~\ref{asdim_zero_group_compact}).
If $H$ is a closed subgroup of  $G$, then $\asdim H \leq \asdim G$~(Corollary \ref{asdim_of_a_closed_subgroup}). 
As a consequence, all discrete subgroups of virtually connected Lie groups have finite asymptotic dimension, whether or not they are finitely generated (Example \ref{Lie_groups}).
We show that the asymptotic dimension of $G$ is the supremum of the asymptotic dimensions of its closed subgroups which have a dense subgroup with a compact set of algebraic generators (Corollary \ref{asdim_by_subgroups_group_compact}).
We also have the following theorem for an extension of Hausdorff topological groups.

\begin{theorem*}[Theorem \ref{extension_of_top_groups}] Let $1\to N\xrightarrow{i}  G\xrightarrow{\pi} Q\to 1$ be an extension of Hausdorff topological groups, where $i$ is a homeomorphism onto its image and every compact subset of $Q$ is the image under $\pi$ of a compact subset of $G$. 
If $\asdim(N)\leq n$ and $\asdim(Q)\leq k$  then $\asdim(G)\leq (n+1)(k+1)-1$.
In particular, if $N$ and $Q$ have finite asymptotic dimension, then $G$ has finite asymptotic dimension.
\end{theorem*}

The {\it free topological group} on a topological space is the analog, in the category of Hausdorff topological groups,  of the free group on a set  in the category of groups.
The free topological group on a non-discrete space is typically {\it not} locally compact (see the discussion following Proposition  \ref{not_locally_compact}).
We show:

\begin{theorem*}[Theorem~\ref{asdim_of _freetop_general}]
If $X$ is a non-empty space which is homeomorphic to a closed subspace of a Cartesian product of metrizable spaces
then the asymptotic dimension of the free topological group on $X$ is $1$.
\end{theorem*}

The paper is organized as follows.  In Section 2 we develop the general theory of compatible coarse structures on a group and apply it to topological groups.
Asymptotic dimension theory in our framework is treated in Section 3.
In Section 4 we compute the asymptotic dimension of a free topological group.

%%%%%%%%%%%%%%%%%%%%%%%%%%%%%%%%%%%%%%%%%%%%
%%
%%  Section: Compatible coarse structures on  a group
%%
%%%%%%%%%%%%%%%%%%%%%%%%%%%%%%%%%%%%%%%%%%%%

\section{Compatible coarse structures on  a group}
\label{sec:compatible_coarse_structures}

In this section we introduce the notion of a {\it compatible coarse structure} on a group $G$ (Definition \ref{def:compatible_coarse}) and show that any such coarse structure on $G$ is obtained from a {\it generating family}, that is, a collection of subsets of $G$ satisfying certain axioms (Definition \ref{generating_family}, Propositions \ref{assoc_coarse} and \ref{assoc_coarse}).
We give several classes of examples of  compatible coarse structures on a group
(Examples
\ref{EX_pseudo_norm}, 
\ref{EX_group_compact},
\ref{EX_restricted_card},
\ref{EX_finite_subsets} and
\ref{EX_bounded}).
Of particular interest is the {\it group-compact} coarse structure on a Hausdorff topological group (Example \ref{EX_group_compact})
and its generalizations (Remark \ref{group_compact_generalized}).
Necessary and sufficient conditions for the group-compact coarse structure on a topological group to coincide with the bounded 
coarse structure  associated to a left invariant metric are given in Theorem \ref{group-compact_equals_metric};
also see Propositions \ref{locally_compact_groups} and \ref{no_left_invariant}.
A characterization of the bounded sets for a group with a compatible coarse structure is given in Proposition \ref{prop_bounded}.
We give a criterion for a surjective homomorphism of groups with compatible coarse structures to be a coarse equivalence
(Corollary \ref{coarse_equiv_cor}) and also a criterion for the inclusion of a subgroup to be a coarse equivalence
(Proposition \ref{coarse_equiv_two}).
These results are applied to Hausdorff topological groups with the group-compact structures
(Propositions \ref{coarse_equiv_compact_normal_subgroup} and  \ref{coarse_equiv_cocompact_subgroup}).

We recall Roe's theory of coarse structures and coarse spaces  (\cite[\S 2]{Roe}).
Let $X$ be a set.   The {\it inverse} of a subset $E$ of $X \times X$, denoted $E^{-1}$,  is the set
\[
E^{-1}   =  \{(y, x) \in X \times X ~{\big |}~  (x, y) \in E \}.
\]
For  subsets  $E_1$ and $E_2$ of $X \times X$, the {\it composition} of $E_1$ and $E_2$, denoted $E_1 \circ E_2$,   is the set
\[
E_1 \circ E_2  =  \{(x, z) \in X \times X ~{\big |}~   \text{there exists } y \in X \text{ such that } (x, y) \in E_1 \text{ and } (y, z) \in E_2\}.
\]

\begin{definition}  (\cite[Definition 2.3]{Roe})
\label{def:coarse_structure}
A {\it coarse structure} on a set $X$ is a collection $\cE$ of subsets of $X \times X$, called {\it entourages},
 satisfying the following properties:
\begin{itemize}

\item[(a)] The diagonal,  $\Delta_X = \{(x, x) ~{\big |}~  x \in X \}$, is an entourage.

\item[(b)]  A subset of an entourage is an entourage.

\item[(c)]  A finite union of entourages is an entourage.

\item[(d)]  The inverse of entourage is an entourage.

\item[(e)]  The composition of two entourages is an entourage.
\end{itemize}
The pair $(X, \cE)$ is called a {\it coarse space}.
\end{definition}

Let $G$ be  a group.   For subsets $A$ and $B$ of  $G$ we write
$AB = \{ ab ~{\big |}~  a \in A \text{ and } b \in B\}$   and
$A^{-1} = \{ a^{-1} ~{\big |}~  a \in A \}$.
The group $G$ acts diagonally on the product $G \times G$ and we say that  $E \subset G \times G$ is {\it $G$-invariant} if  $GE = E$
where
$GE = \{ (ga, gb) ~{\big |}~   (a,b) \in E \text{ and } g \in G \}$.

\begin{definition}
\label{def:compatible_coarse}
A  coarse structure $\cE$ on a group $G$ is {\it compatible} if every entourage is contained in a $G$-invariant entourage.
\end{definition}

We describe a method of obtaining compatible coarse structures on a given group $G$.

\begin{definition}
\label{generating_family}
A family $\cF$  of subsets of $G$ is a {\it generating family for a compatible coarse structure on $G$}  (abbreviated as ``generating family on $G$'')
if it has the following properties:
\begin{itemize}

\item[(a)]  There exists $A \in \cF$ which is non-empty.

\item[(b)]  A finite union of  elements of $\cF$ is in $\cF$.

\item[(c)]  If $A$ and $B$ are in $\cF$ then $AB$ is in $\cF$.

\item[(d)]  If $A$ is in $\cF$ then $A^{-1}$ is in $\cF$.
\end{itemize}

\end{definition}

Our terminology is justified by the following propositions.

\begin{proposition}
\label{assoc_coarse}
Let $\cF$ be a generating family on $G$ as in Definition \ref{generating_family}.   Define
\[
\cE_{\cF} = \{ E \subset G \times G ~{\big |}~  \text{there exists } A \in \cF \text{ such that }  E \subset G(A \times A) \}.
\]
Then  $\cE_{\cF}$ is a compatible coarse structure on $G$.
\end{proposition}

We say that $\cE_{\cF}$ is {\it the coarse structure associated to $\cF$}.

\begin{proof}
If $A \in \cF$ is non-empty then $\Delta_G \subset G(A \times A)$ and so $\Delta_G \in \cE_{\cF}$. 
If $A, B \in \cF$ then  $G(A \times A) \cup G(B \times B)   \subset G( (A\cup B) \times (A \cup B))$ which implies that the union of two elements of $\cE_{\cF}$  is in $\cE_{\cF}$.
Observe that if $A, B \in \cF$ then  $G(A \times B) \in \cE_{\cF}$ because  $A \cup B \in \cF$ and  $G(A \times B) \subset G( (A\cup B) \times (A \cup B))$.
The composition of two element in $\cE_{\cF}$  is in $\cE_{\cF}$ because
for $A, B  \in \cF$ we have
$G(A \times A) \circ G(B \times B) \subset G(A \times (A B^{-1} B))$  and  $A B^{-1} B \in \cF$ by properties (c) and (b) in Definition \ref{generating_family}.
Hence $\cE_{\cF}$ is a coarse structure and, by definition, is compatible.
\end{proof}

We show that every compatible coarse structure $\cE$ on a group $G$ is of the form $\cE_{\cF}$  for some generating family $\cF$ on $G$.
For any group $G$ the {\it shear map}, $\pi_G \colon G \times G ~\ra~ G$, is defined by  $\pi_G(x,y) = y^{-1}x$.

\begin{proposition}
\label{every_coarse}
Let $\cE$ be a compatible coarse structure on a group $G$. 
Let $\pi_G\colon G \times G ~\ra~ G$ be the shear map.
Define  $\cF(\cE) = \{ \pi_G(E) ~{\big |}~  E \in \cE \}$.
Then  $\cF(\cE)$ is a generating family on $G$ and
$\cE = \cE_{\cF(\cE)}$.
\end{proposition}

\begin{proof}
We first show that $\cF(\cE)$ is a generating family on $G$, that is,
properties (a) through (d) of Definition \ref{generating_family} hold for  $\cF(\cE)$.
Property (a) is obvious.
Property (b) follows from the equality $\pi_G(E) \cup \pi_G(E') = \pi_G(E  \cup E')$.
Assume that $A \subset \pi_G(E)$ and $B \subset \pi_G(E')$  where $E, E' \in \cE$ are $G$-invariant.
We claim that $AB \subset \pi_G(E' \circ E)$ from which it follows that
$AB = \pi_G(\pi_G^{-1}(AB) \cap (E' \circ E)) \in \cF(\cE)$.
Let $a = y^{-1}x \in A$ where  $(x,y) \in E$ and
$b = v^{-1}u \in B$ where  $(u,v) \in E'$.
Since $E$ and $E'$ are $G$-invariant,  we have
$(1, x^{-1}y) = x^{-1}(x,y) \in E$ and
$(v^{-1}u,1) = v^{-1}(u,v) \in E'$.
Hence $(v^{-1}u, x^{-1}y)  \in E' \circ E$ and so
$ab = (x^{-1}y)^{-1}v^{-1}u \in \pi_G(E' \circ E)$, verifying the claim.
If $A=\pi_G(E)$ then $A^{-1} = \pi_G(E^{-1})$ and so property (d) holds.

By its definition,
\[
\cE_{\cF(\cE)} = \{ P \subset G \times G ~{\big |}~  \text{there exists } E \in \cE \text{ such that }  P \subset G(\pi_G(E) \times \pi_G(E)) \}.
\]
Observe that $\{1\} \in \cF(\cE)$ because $\pi_G(\Delta_G) = \{1\}$.
For any $E \in \cE$ and $(x,y) \in E$ we have
$(x,y) = y(y^{-1}x, 1) \in G(\pi_G(E) \times \{1\}) \in  \cE_{\cF(\cE)}$ which shows that $\cE \subset \cE_{\cF(\cE)}$.

Let $E \in \cE$ be $G$-invariant.
We claim that $G(\pi_G(E) \times \pi_G(E))  \subset E \circ E^{-1}$.   
Since $E \circ E^{-1} \in \cE$ and any entourage in $\cE$ is contained in a $G$-invariant entourage,
this would imply that $\cE_{\cF(\cE)}  \subset  \cE$.
Let $(a,b) \in \pi_G(E) \times \pi_G(E)$.
Then $a = y^{-1}x$ and $b=v^{-1}u$ where $(x,y), ~(u,v) \in E$.
We have that $(y^{-1}x, 1) = y^{-1}(x,y) \in E$ and  $(v^{-1}u, 1) = v^{-1}(u,v) \in E$.
Since $(1, v^{-1}u) \in E^{-1}$, it follows that $(a,b) = (y^{-1}x, v^{-1}u) \in E \circ E^{-1}$.
Hence \linebreak
$\pi_G(E) \times \pi_G(E) \subset E \circ E^{-1}$ which
verifies the claim since $E \circ E^{-1}$ is $G$-invariant.
\end{proof}

\begin{definition}
Let $\cF$ be a generating family on a group $G$.  Define the {\it completion of $\cF$},  denoted by $\widehat{\cF}$, to be the collection of subsets of $G$ given by
\[
\widehat{\cF} = \{ A \subset G ~{\big |}~   \text{there exists } B \in \cF \text{ such that } A \subset B \}.
\]
\end{definition}

It is clear that the completion of a generating family is a generating family.

\begin{proposition}
\label{prop_genfamtwo}
Let $\cF$ be a generating family on a group $G$. 
Then
$\widehat{\cF} = \cF(\cE_{\cF})$ and $\cE_{\cF} = \cE_{\widehat{\cF}}$.
\end{proposition}

\begin{proof}
By its definition,
\[ 
 \cF(\cE_{\cF}) = \{ \pi_G(E) \subset G  ~{\big |}~  \text{ there exists }  B \in \cF \text{ such that } E \subset G(B \times B) \}.
 \]
Let $A \in \widehat{\cF}$ be non-empty.
There exists $B \in \cF$ such that $A \subset B$.
Note that $B \cup B^{-1}B \in \cF$.
We have,
$
A \times \{1\}  \subset B \times \{1\}  \subset G((B \cup B^{-1}B) \times (B \cup B^{-1}B))
$
and so  $A = \pi_G(A \times \{1\}  ) \in \cF(\cE_{\cF})$.
Hence $\widehat{\cF} \subset \cF(\cE_{\cF})$.
Let $A \in  \cF(\cE_{\cF})$.
Then there exists $B \in \cF$ and $E \subset G(B \times B)$ such that $A = \pi_G(E)$.
We have $A \subset \pi_G(G(B \times B)) = B^{-1}B \in \cF$ and so $A \in \widehat{\cF}$.
Hence $\cF(\cE_{\cF}) \subset \widehat{\cF}$.
We conclude that $\cF(\cE_{\cF}) = \widehat{\cF}$ and so by
Proposition \ref{every_coarse},
$\cE_{\cF} = \cE_{\cF(\cE_{\cF}) } = \cE_{\widehat{\cF}}$.
\end{proof}

\begin{corollary}
\label{same_coarse}
Let ${\cF}_1$  and ${\cF}_2$ be a generating families on a group $G$.  
Then $\cE_{{\cF}_1}  = \cE_{{\cF}_2}$  if and only if $\widehat{\cF}_1 = \widehat{\cF}_2$. \qed
\end{corollary}

We give some examples of generating families and their associated coarse structures.

\begin{example}[Pseudo-norms on groups]
\label{EX_pseudo_norm}
A {\it pseudo-norm} on a group $G$ is a non-negative function  $| \cdot | \colon G ~\ra~ \R$
such that:
\begin{enumerate}
\item $| 1 | = 0$,

\item For all $x \in G$, $|x^{-1} | = | x|$,

\item  For all $x, y \in G$,  $| x y| \leq |x| + |y|$.
\end{enumerate}
A pseudo-norm on $G$ determines a left invariant pseudo-metric $d$ on $G$ given by  $d(x,y) = | y^{-1} x|$.
Conversely, any left invariant pseudo-metric $d$ on $G$ yields a pseudo-norm given by $|x | = d(x,1)$.
For a non-negative real number $r$, 
let $B(r) =\{ x \in G ~{\big |}~  |x| \leq r\}$,  the {\it closed ball of radius $r$ centered at $1 \in G$}.
Define
\[
{\cF}_d = \{ A \subset G ~{\big |}~  \text{ there exists }  r>0 \text{ such that } A \subset B(r)\}.
\]
Thus ${\cF}_d$ consists  of those subsets of $G$ which are bounded with respect to the pseudo-norm.
Since 
$B(r) \cup B(s) = B(\max\{r,s\})$,
$B(r)^{-1} = B(r)$  and $B(r) B(s) \subset B(r+s)$,  it follows that ${\cF}_d$ is a generating family on $G$.
Note that $\widehat{{\cF}_d} =  {\cF}_d$.
The  coarse structure $\cE_{{\cF}_d}$ (henceforth abbreviated as $\cE_{d}$) is called the {\it bounded coarse structure associated to the pseudo-metric $d$} and
\[
\cE_{d} = \{  E \subset G \times G ~{\big |}~    \sup\{d(x,y) ~{\big |}~  (x,y) \in E\}  <  \infty \}.
\]

\end{example}

\begin{example}[The group-compact coarse structure]
\label{EX_group_compact}
Let $G$ be a Hausdorff topological group  and let $\cC(G)$ be the collection of all compact subsets
of  $G$.   If  $K$ and $K'$ are compact subsets of $G$ then  $K \cup K'$ is compact and
the continuity of the group operations implies that $K^{-1}$ and $K K'$ are compact.
It follows that $\cC(G)$ is a generating family on $G$ and
\[
\cE_{\cC(G)} = \{  E \subset G \times G ~{\big |}~    \text{ there exists a compact subset }  K \subset G \text{ such that }  E \subset G(K  \times K) \}.
\]
We call this coarse structure the {\it group-compact coarse structure} on $G$.
\end{example}

\begin{remark}[Generalizations of the group-compact coarse structure]   \label{group_compact_generalized}  
$ $  %WARNING TeX hack

\noindent(1)
Let $G$ be a  topological group which is  not necessarily Hausdorff.
The collection of all quasi-compact subsets of $G$ is a generating family on $G$ (recall that $A \subset G$ is quasi-compact if every open cover of $A$  has a finite subcover).

\smallskip

\noindent(2)
A less restrictive notion of a Hausdorff topological group is obtained replacing the requirement that the group multiplication $\mu \colon G \times G  ~\ra~ G$ is continuous, where $G \times G$ has the product topology, 
with the condition that $\mu$ is continuous when the set $G \times G$ is given the weak topology determined by the collection of compact subsets of the space $G \times G$. 
For the purpose of this discussion, we say that $G$ is a {\it weak Hausdorff topological group}.
A natural example of a weak Hausdorff topological group is the geometric realization
of a simplicial group. The collection $\cC(G)$ of compact subsets of  a weak Hausdorff topological group $G$ is a generating family on $G$.
If $G_{\rm k}$ is the weak Hausdorff topological  group obtained by re-topologizing $G$ with the weak topology determined by its collection of compact subsets 
then $\cC(G)= \cC(G_{\rm k})$ so the corresponding group-compact coarse structures on the underlying abstract group are the same.

\smallskip

\noindent(3)  Let $G$ be a Hausdorff topological group and let $X$ be a Hausdorff space equipped with a continuous left action of $G$.
Assume that $X = GC$ for some compact $C \subset X$ and that the $G$-action is {\it proper},  that is, the map $A \colon G \times X ~\ra~X \times X$ given by $A(g,x) = (x, gx)$  is  a proper map 
(recall that a continuous map between Hausdorff spaces is proper if it is a closed map and the fibers are compact).

The {\it group-compact coarse structure} on $X$ is the coarse structure:
\[
\cE_{\text{$G$-cpt}}=\{ E \subset X \times X~{\big |}~  \text{there exists  a compact } K \subset X  \text{ such that }  E \subset G(K \times K) \}.
\]
When $X=G$ with the left translation action of $G$ on itself,  this construction recovers the coarse structure $\cE_{\cC(G)}$ on $G$.
Another case of interest in this paper is the homogeneous space $X = G/K$ where is $K$ is a compact subgroup of $G$
(see Remark \ref{not_necessarily_normal}   and Example \ref{Lie_groups}).
\end{remark}

\begin{example}[Subsets of restricted cardinality]
\label{EX_restricted_card}
Let $G$ be a group and $\kappa$ an infinite cardinal number.   
Let $\cF_\kappa$ be the collection of all subsets of $G$ of cardinality strictly less than that of $\kappa$.
Then $\cF_\kappa$ is a generating family on $G$.
\end{example}

\begin{example}[The group-finite coarse structure]
\label{EX_finite_subsets}
Let $G$ be a group and let $\operatorname{fin}(G)$ be the collection of all finite subsets of $G$.  
Then $\operatorname{fin}(G)$ is a generating family on $G$,  indeed it is a special case of each of the three
preceding examples.
If $G$ is given the discrete topology then  $\operatorname{fin}(G) = \cC(G)$ since the compact subsets of $G$ are precisely the finite subsets.
If $\kappa$ is the first infinite cardinal then  $\operatorname{fin}(G) = \cF_\kappa$.
In the case $G$ is countable, if $d$ is a weighted word metric associated to some (possibly infinite) set of generators of $G$  as in
\cite[Proposition 1.3]{Dranish_and_Smith}
then  $\operatorname{fin}(G) = {\cF}_d$
(see \cite[Remark 2]{Dranish_and_Smith}).

We call the coarse structure $\cE_{\operatorname{fin}(G)}$ the {\it group-finite coarse structure} on $G$.
\end{example}

\begin{example}[Topologically bounded sets]
\label{EX_bounded}
Let $G$ be a topological group.
A subset $B$ of $G$ is said to be {\it topologically bounded} if for every neighborhood $V$ of $1 \in G$ there exists a positive integer $n$ (depending on $V$) such that 
$B \subset V^n = V \cdots V$ ($n$ factors).
The collection $\cF_{\rm tbd}$ of all topologically bounded subsets of $G$ is easily seen to be a generating family on $G$.  
If $d$ is a left invariant  pseudo-metric $d$ inducing the topology of $G$ then any topologically bounded set is contained in a $d$-ball centered at $1$ and so $\cF_{\rm tbd} \subset \cF_d$; however, 
the inclusion  $\cF_d \subset \cF_{\rm tbd}$ is not, in general, valid without additional assumptions on $d$.
\end{example}

A compatible coarse structure on a group determines compatible coarse structures on its subgroups and quotient groups.

\begin{proposition}[Subgroups and quotient groups]
\label{sub_and_quotient}
Let $G$ be a group and $\cF$ a generating family on $G$.
\begin{itemize}

\item[(i)]  Let $H  \subset  G$ be a subgroup.  Assume that there exists $A \in \cF$ such that $A \cap H$ is non-empty.
Then  the collection,  $\cF|_H$,  of subsets of $H$ given by
$
\cF|_H = \{ A \cap H ~\big|~ A \in \cF \}
$
is a generating family on $H$.

\item[(ii)]  Let  $\phi \colon G ~\ra~ Q$ be a homomorphism.
Then  the collection,  $\phi( \cF)$,  of subsets of $Q$ given by
$
\phi( \cF) = \{ \phi(A) ~\big|~ A \in \cF \}
$
is a generating family on $Q$.

\end{itemize}
\end{proposition}

We omit the straightforward proof.

\begin{proposition}[Enlargement by a  normal subgroup]
\label{normal_subgroups}
Let $G$ be a group,  $N \lhd  G$ a normal subgroup and $\cF$ a generating family on $G$.
Define $N \cF$ to be the collection of subsets of $G$ given by
$
N \cF = \{ NA ~\big|~ A \in \cF \}.
$
Then $N \cF$ is a generating family on $G$. 
\end{proposition} 

\begin{proof}
If $A \in \cF$ is nonempty then so is $NA \in N\cF$.
For $A, B \in \cF$ we have  $NA \cup NB = N(A \cup B) \in N\cF$ because $A \cup B \in \cF$.
Since $N$ is a normal subgroup of $G$, for any $X \subset G$ we have $NX = XN$.
Hence for  $A, B \in \cF$ we have $(NA)(NB) = (NN)(AB) = N(AB) \in N\cF$ because $AB \in \cF$.
Also, $(NA)^{-1} = A^{-1} N^{-1} = A^{-1} N = N A^{-1} \in N\cF$ because \linebreak
$A^{-1} \in \cF$. 
\end{proof}

\begin{theorem}
\label{group-compact_equals_metric}
Let $G$ be a Hausdorff topological group.  Denote its topology by $\tau$.  Let $d$ be a left invariant pseudo-metric on $G$ (not necessarily inducing the topology $\tau$).
Then the group-compact coarse structure on $G$ (arising from the topology $\tau$) coincides
with the bounded coarse structure associated to $d$
if and only if:
\begin{itemize}

\item[(i)]  every relatively compact subset of $G$ (with respect to $\tau$) is $d$-bounded,

\item[(ii)]  every $d$-bounded subset of $G$ is relatively compact (with respect to $\tau$).

\end{itemize}
\end{theorem}

\begin{proof}
Conditions (i) and (ii) are equivalent to $\widehat{{\cF}_d} =   \widehat{\cC(G)}$ and so the conclusion follows from
Corollary \ref{same_coarse}.
\end{proof}

Let $G$ be a group and $\Sigma \subset G$ a (not necessarily finite) set of generators.
The {\it word length norm} associated to $\Sigma$, denoted by $| x |_\Sigma$ for $x \in G$, is defined by
\[
| x |_\Sigma =  \inf\{  n ~{\big |}~   x = a_1 \cdots a_n, \text{ where }  a_i \in \Sigma \cup \Sigma^{-1} \}.
\]
We denote the associated word length metric by $d_\Sigma$.

\begin{proposition}
\label{locally_compact_groups}
Let $G$ be a locally compact group with a set of generators $\Sigma \subset G$ such that $\Sigma \cup \{1\} \cup \Sigma^{-1}$ is compact.
Then the group-compact coarse structure on $G$ coincides with the bounded coarse structure associated to $d_\Sigma$.
\end{proposition}

\begin{proof}
By \cite[Lemma 3.2]{Abels},  every compact subset of $G$ has finite word length (with respect to the generating set $\Sigma$) so Condition (i)
of Theorem \ref{group-compact_equals_metric} holds.
The $d_\Sigma$-ball of non-negative integer radius $n$ is $(\Sigma \cup \{1\} \cup \Sigma^{-1})^n$, which is compact since $\Sigma \cup \{1\} \cup \Sigma^{-1}$ is assumed to be
compact,  hence Condition (ii) of Theorem \ref{group-compact_equals_metric} also holds.
\end{proof}

\begin{example}
%\label{}
The Lie group $\R^m$ is locally compact and $\Sigma = [-1,1]^m$ is a compact set of generators.  
Hence the group-compact coarse structure coincides with the bounded coarse structure associated to $d_\Sigma$.
Note that the Euclidean metric on $\R^m$  also satisfies Conditions (i) and (ii)
of Theorem \ref{group-compact_equals_metric}  as does any appropriate ``coarse path pseudo-metric'' 
(see \cite[Proposition 3.11]{Abels}).
\end{example}

A topological group which is not locally compact may fail to have a left invariant pseudo-metric such that the associated bounded coarse structure coincides
with the group-compact coarse structure.
We show that  this is the case for the additive group $\Zhalf$ of rational numbers whose denominators are powers of two, topologized as a subspace of $\R$
(and, as such, is not locally compact).

\begin{proposition}
\label{no_left_invariant}
There is no left invariant pseudo-metric on the topological group $\Zhalf$  such that the associated bounded coarse structure coincides with the group-compact
coarse structure.
\end{proposition}

\begin{proof}
Let $d_E(x,y) = | x - y|$, the Euclidean absolute value of $x-y$.
Clearly, any compact subset of $\Zhalf$ has bounded Euclidean absolute value and so $\cC(G) \subset \cF_{d_E}$.
The ball $B_{d_E}(1) = \Zhalf \cap [-1,1]$ is not a relatively compact subset of $\Zhalf$ (for example, the sequence $x_n = (1 - 4^{-n})/3 \in B_{d_E}(1)$
has no convergent subsequence in $\Zhalf$).
Thus the bounded coarse structure on $\Zhalf$ associated to $d_E$ is strictly coarser than the group-compact coarse structure.
Suppose that $d$ is a left invariant pseudo-metric on $\Zhalf$ such that the associated coarse structure coincides with the group-compact coarse structure.
By Theorem \ref{group-compact_equals_metric},  $d$ must satisfy Conditions  (i) and (ii) of that proposition.
We will show that these conditions on $d$ imply that the bounded coarse structure associated to $d$ coincides with
the bounded coarse structure associated to $d_E$, a contradiction, thus proving that no such $d$ exists.

Assume that $d$ satisfies Conditions  (i) and (ii) of Theorem \ref{group-compact_equals_metric}.
By Condition (ii),   the $d$-ball $B_d(r) = \{  x \in \Zhalf ~{\big |}~   |x|_d = d(x,0) \leq r\}$ is relatively compact as a subset of $\Zhalf$ and hence also as a subset of $\R$.
It follows that $\sup\{ |x| ~{\big |}~  x \in B_d(r)\}$ is finite and so ${\cF}_d \subset {\cF}_{d_E}$.

Let $K = \{ \pm  2^{-k} ~{\big |}~   k=0, 1, \ldots \} \cup \{ 0\}$. 
Note that $K$ is a compact set of generators for $\Zhalf$. 
Let $\lambda K$, where $\lambda \in \R$,  denote the set $\{ \lambda x ~{\big |}~  x \in  K\}$.
For $n \geq 0$,  let
\[
F_n = K + 2^{-1}K + \cdots  + 2^{-n}K = \{ x \in \Zhalf ~{\big |}~   x= \sum_{i=0}^n  2^{-i}a_i,\text{ where }  a_i \in K \}.
\]
Since $K$ is compact so is each $F_n$.
Let $\phi\colon \bN ~\ra~  \bN$,  where $\bN$ is the set of non-negative integers, be any function such that $\lim_{n \ra \infty} \phi(n) = \infty$.
Let $F_\phi = \bigcup_{n \geq 0} 2^{-\phi(n)} F_n$.
Observe that $F_\phi \subset \Zhalf$ and that $F_\phi$ is compact because $0 \in 2^{-\phi(n)} F_n$ and the Euclidean diameters
of the compact sets $2^{-\phi(n)} F_n$ converge to $0$.  
By Condition (i),  $F_\phi$ is $d$-bounded, that is,
there exists $C_\phi > 0$ such that $|x|_d \leq C_\phi$ for all $x\in F_\phi$.
Let $x \in B_{d_E}(1)$.
Observe that $x \in F_n$ where  $n= |x|_K$ (recall that $|x|_K$ is the word length norm of $x$ with respect to the generating set $K$).
It follows that  $2^{-\phi(n)} x \in F_\phi$ and so $|2^{-\phi(n)} x |_d \leq  C_\phi$.  Hence
\begin{equation}
\label{d_versus_K}
|x|_d =  |  2^{\phi(n)} 2^{-\phi(n)} x|_d  \leq 2^{\phi(n)} |2^{-\phi(n)} x|_d  \leq 2^{\phi(n)} C_\phi ~=~ C_\phi ~2^{\phi( |x|_K)}.
\end{equation}
Suppose  $B_{d_E}(1)$ is not $d$-bounded.
Then there exists a  sequence $\{x_n \} \subset B_{d_E}(1) $ such that  $|x_n|_d ~\ra~ \infty$.  
Choosing $\phi$ to be the identity function in (\ref{d_versus_K}),  we see that $|x_n|_K ~\ra~ \infty$ and by passing to a subsequence
we may assume that $\{|x_n|_d \}$ and $\{|x_n|_K \}$ are strictly increasing.
For a real number $r$, let $r^+$ denote the smallest integer greater than or equal to $r$.  
Define $\phi\colon \bN ~\ra~  \bN$ on  $\{|x_n|_K \} \subset \bN$  by 
$\phi(|x_n|_K) = (\tfrac{1}{2} \log_2(|x_n|_d))^+$,
where $\log_2$ is the base two logarithm,
and extend it to all of $\bN$  
so that $\phi$ is non-decreasing.
For this $\phi$, (\ref{d_versus_K}) yields:
\[
|x_n|_d ~\leq~ C_\phi ~2^{\phi( |x_n|_K)}  ~\leq~ 2 \, C_\phi \, |x_n|_d^{1/2}.
\]
It follows that  $|x_n|_d$ is bounded, a contradiction.
Hence $B_{d_E}(1) $ is $d$-bounded and so
$C= \sup\{ |x|_d ~{\big |}~  x \in  B_{d_E}(1)\}$ is finite.
If $m$ is a positive integer and $x \in  B_{d_E}(m)$ then
\[
|x|_d  = | m(x/m) |_d \leq m |x/m|_d   \leq m C
\]
and so $B_{d_E}(m) \subset B_d(mC)$.   
It follows that    ${\cF}_{d_E}  \subset {\cF}_d$.
We have established that ${\cF}_{d_E} = {\cF}_d$ and so $d$ and $d_E$ give rise to the same coarse structure on $\Zhalf$.
\end{proof}

A coarse space $(X, \cE)$ is said to be {\it connected}  if every point of $X \times X$ is contained in some entourage.

\begin{proposition}
\label{prop_connected}
Let $G$ be a group and $\cF$ a generating family on $G$.   Then $(G, \cE_{\cF})$  is connected if and only if $G = \bigcup_{A \in \cF} A$.
\end{proposition}

\begin{proof}
Assume  that $(G, \cE_{\cF})$  is connected and let $g \in G$.   Then for some $B \in \cF$ we have $(g,1) \in G(B \times B)$.
It follows that $g \in B^{-1} B \in \cF$.   Hence $G = \bigcup_{A \in \cF} A$.

For the converse, assume $G = \bigcup_{A \in \cF} A$  and let $(x,y) \in G \times G$.
There exist $A, B \in \cF$ such that $x \in A$ and $y \in B$.   Let $C = A \cup B$.
Then $(x, y) \in G(C \times C) \in \cE_{\cF}$.
\end{proof}

\begin{corollary}
\label{cor_connected}
If the coarse space  $(G, \cE_{\cF})$  is connected  then for all $A \in \widehat{\cF}$  and all $g \in G$ we have
that $g A \in \widehat{\cF}$ and $ A g \in \widehat{\cF}$ 
\end{corollary}

\begin{proof}
If $(G, \cE_{\cF})$  is connected then Proposition \ref{prop_connected} implies that $\{g \}  \in \widehat{\cF}$ and
so $gA = \{g \}  A \in \widehat{\cF}$ and $ Ag = A \{g \}   \in \widehat{\cF}$.
\end{proof}
It is straightforward to show, using Proposition \ref{prop_connected},  that
the coarse spaces in Examples
\ref{EX_pseudo_norm}, 
\ref{EX_group_compact} and
\ref{EX_restricted_card}
are connected.

For a set $X$ and $E \subset X \times X$ and $x \in X$,  
let $E(x)$ denote the set $\{y \in X ~\big|~ (y,x) \in E \}$.
A subset $B \subset X$ of a coarse space $(X, \cE)$ is said to be {\it bounded} if it is of the form $E(x)$ for some $E \in \cE$ and $x \in X$.

\begin{proposition}
\label{prop_bounded}
Let $G$ be a group and $\cF$ a generating family on $G$. 
A subset $B \subset G$ is bounded (with respect to the coarse structure $ \cE_{\cF}$) if and only if $B^{-1}B \in \widehat{\cF}$.
Every element of $\widehat{\cF}$ is bounded and 
if $(G, \cE_{\cF})$  is connected then  $\widehat{\cF}$ coincides with the collection of bounded subsets of $G$.
\end{proposition}

\begin{proof}
Any bounded subset $B \subset G$ is a subset of a set of the form
$C=G(A \times A)(x) =  xA^{-1}A$ where $x \in X$ and $A \in \cF$.   
Observe that $B^{-1}B  \subset C^{-1}C= A^{-1}AA^{-1}A \in \cF$ and so $B^{-1}B \in \widehat{\cF}$.

Assume that $B \subset G$ is nonempty and $B^{-1} B \in  \widehat{\cF}$.
Observe that $\pi_G(G(B \times B)) = B^{-1} B \in  \widehat{\cF}$ where $\pi_G$ is the shear map.
By Proposition \ref{prop_genfamtwo}, $G(B \times B) \in \cE_{\cF}$ and so $G(B \times \{b\})  \in  \cE_{\cF}$ for $b \in B$.
Since $B = G(B \times \{b\})(b)$ it follows that $B$ is bounded.

If $A \in  \widehat{\cF}$ then $A^{-1}A \in  \widehat{\cF}$ and thus  $A$ is bounded.
Assume that $(G, \cE_{\cF})$  is connected and let $B \subset G$ be bounded (and nonempty).
Since $G(B \times B) \in \cE_{\cF}$, we have $G(B \times \{b\})  \in  \cE_{\cF}$ for $b \in B$.
Hence  $\pi_G(G(B \times \{b\})) = b^{-1}B \in  \widehat{\cF}$.
By Corollary \ref{cor_connected}, $B = b (b^{-1}B) \in   \widehat{\cF}$.
\end{proof}

Next, we consider morphisms between coarse spaces 
(\cite[\S 2] {Roe} is a general reference for this topic).

\begin{definition}
Let $(X, \cE)$ and $(Y,\cE')$ be coarse spaces and let $f \colon X ~\ra~Y$ be a map.
\begin{enumerate}

\item  The map $f$ is {\it coarsely uniform} (synonymously, {\it bornologous})  if for all $E \in \cE$,   \newline
           $(f \times f)(E) \in \cE'$.

\item  The map $f$ is {\it coarsely proper} if the preimage of any bounded set in $Y$ is a bounded set in $X$.

\item The map $f$ is a {\it coarse embedding} if it is coarsely uniform and for all $E \in \cE'$,       $(f \times f)^{-1}(E) \in \cE$.

\end{enumerate}
\end{definition}

\begin{proposition}
\label{morphisms}
Let $G$  and $H$ be groups and  let $\cF$ and $\cF'$ be generating families on $G$ and $H$ respectively.
Let $G$ and $H$ have the compatible coarse structures $\cE_{\cF}$ and $\cE_{\cF'}$ respectively. 
Let $f \colon G ~\ra~ H$ be a homomorphism.

\begin{enumerate}
\item  The map $f$ is coarsely uniform if and only if  for all $F \in \cF$,  $f(F) \in  \widehat{\cF'}$.

\item  If for all $F' \in \cF'$, $f^{-1}(F') \in \widehat{\cF}$  then $f$ is coarsely proper.

\item  If  for all $F \in \cF$,  $f(F) \in \widehat{\cF'}$ and for all $F' \in \cF'$, $f^{-1}(F') \in  \widehat{\cF}$ 
          then $f$ is a coarse embedding.

\end{enumerate}

\end{proposition}

\begin{proof}
Let $\pi_G \colon G \times G ~\ra~ G$ and  $\pi_H \colon H \times H ~\ra~ H$ be the respective shear maps.
Assertion (1) of the Proposition follows from the  identity 
$\pi_H((f \times f)(E)) = f(\pi_G(E))$, which is valid for any $E \subset G \times G$ and in particular for $E \in \cE_{\cF}$,  and Proposition \ref{prop_genfamtwo}.
Assertion (2) follows from the inclusion $(f^{-1}(B))^{-1}f^{-1}(B) \subset  f^{-1}(B^{-1}B)$, which is valid for any $B \subset H$ and in particular for bounded subsets of $H$,
and Proposition \ref{prop_bounded}.
Assertion (3) follows from the inclusion $\pi_G((f \times f)^{-1}(E')) \subset f^{-1}(\pi_H(E'))$, 
which is valid for any $E' \subset H \times H$ and in particular for $E' \in \cE_{\cF'}$,  and Proposition \ref{prop_genfamtwo}.
\end{proof}

\begin{definition}[Coarse equivalance]  $\null$

\begin{enumerate}
\item  Let $(X, \cE)$ be a coarse space and let $S$ be a set.  Two maps  $p, q\colon S ~\ra X$ are {\it close} if
$\{ (p(s), q(s)) ~\big|~ s \in S\} \in \cE$.

\item  Let $(X, \cE)$ and $(Y,\cE')$ be coarse spaces.  A coarsely uniform map $f \colon X ~\ra~Y$ is a {\it coarse equivalence} if there exists a
coarsely uniform map $\psi \colon Y ~\ra~ X$ such that $\psi \circ f$ is close to the identity map of $X$ and $f \circ \psi$ is close to the identity map of $Y$.
The map $\psi $ is called a {\it coarse inverse} of  $f$.

\end{enumerate}

\end{definition}

The following criterion for a coarse embedding to a be coarse equivalence will be useful.

\begin{lemma}
\label{coarse_embedding_lemma}
Let $f\colon (X , \cE) ~\ra~ (Y, \cE')$ be a coarse embedding.
If $\psi \colon Y ~\ra X$ is a map such that  $f \circ \psi$ is close to the identity of $Y$ then
$f$ is a coarse equivalence.
\end{lemma}

\begin{proof}
Let $\psi \colon Y ~\ra X$ be a map  such that $f \circ \psi$ is close to the identity of $Y$.
Then $M = \{ (f(\psi(y)), y) ~|~ y \in Y \}$ is in $\cE'$. 
Let $E' \in \cE'$.
Then  $(f \times f) ( (\psi \times \psi)(E')) = M \circ E' \circ M^{-1}  \in \cE'$.
Since $f$ is coarse embedding,  $(f \times f)^{-1}((f \times f) ( (\psi \times \psi)(E'))) \in \cE$.
It follows that $(\psi \times \psi)(E') \in \cE$ because
$(\psi \times \psi)(E') \subset (f \times f)^{-1}((f \times f) ( (\psi \times \psi)(E')))$.
Hence $\psi$ is coarsely uniform.

Let $P = \{ (\psi(f(x)), x) ~|~ x \in X \}$.
Note that $(f \times f)(P) = M \circ (f \times f)(\Delta_X) \in \cE'$.
Since $f$ is a coarse embedding,  $(f \times f)^{-1}( (f \times f)(P)     \in \cE$ 
and so  $P \subset  (f \times f)^{-1}( (f \times f)(P)$ is also in $\cE$.
Thus $\psi \circ f$ is close to the identity of $X$.
\end{proof}

Note that Lemma \ref{coarse_embedding_lemma} implies that a surjective coarse embedding $f \colon (X, \cE) ~\ra~ (Y,\cE')$ is a coarse equivalence;  a coarse inverse of $f$ is given by  any section of $f$, that is,
a map $s \colon Y ~\ra~X$ such that $f \circ s$ is the identity map of $Y$.

\begin{proposition}
\label{coarse_equiv_one}
Let $G$ be a group and $\cF$ a generating family on $G$.
Let $\phi \colon G ~\ra~ Q$ be a surjective homomorphism. 
Let $N = \ker(\phi)$.
Then $\phi \colon (G, \cE_{N\cF}) ~\ra~ (Q, \cE_{\phi(\cF)})$ is a coarse equivalence
 (see Propositions  {\rm \ref{normal_subgroups}} and {\rm \ref{sub_and_quotient}} for the definitions of $N\cF$ and $\phi(\cF)$, respectively).
\end{proposition}

\begin{proof}
By Proposition \ref{morphisms}(3), $\phi$ is a coarse embedding and thus a coarse equivalence since it is surjective by hypothesis.
\end{proof}

\begin{corollary}
\label{coarse_equiv_cor}
Let $G$ be a group and $\cF$ a generating family on $G$.
Let $\phi \colon G ~\ra~ Q$ be a surjective homomorphism. 
If  $\ker(\phi) \in {\widehat \cF}$ then $\phi \colon (G, \cE_{\cF}) ~\ra~ (Q, \cE_{\phi(\cF)})$ is a coarse equivalence
\end{corollary}

\begin{proof}
If  $N = \ker(\phi) \in {\widehat \cF}$ then $ {\widehat {N\cF}} =  {\widehat \cF}$.  
Also note that  $\phi({\widehat \cF})= \widehat{\phi(\cF)}.$
The conclusion of the Corollary follows from Proposition \ref{coarse_equiv_one} and  Proposition \ref{prop_genfamtwo}.
\end{proof}

\begin{proposition}
\label{coarse_equiv_two}
Let $G$ be a group and $\cF$ a generating family on $G$.
Let $H \subset G$ be a subgroup.
Assume that there exists $B \in \cF$ such that  $HB = G$.
Then the inclusion map $i \colon (H, \cE_{\cF|_H}) ~\ra~ (G, \cE_{\cF})$ is a coarse equivalence.
\end{proposition}

\begin{proof}
The set $HB$ can be expressed as the disjoint union of right cosets of $H$ with coset representatives in $B$ and so
$G = \coprod_{j \in J} Hb_j$, where $\{ b_j ~\big|~ j \in J\} \subset B$.
Since $1 = h_0 b_{j_0}$ for some $h_0 \in H$ and $j_0 \in J$, we have $H \cap B^{-1}$ is non-empty (and so $\cF|_H$ is non-empty).
Clearly, $i \colon (H, \cE_{\cF|_H}) ~\ra~ (G, \cE_{\cF})$ is a coarse embedding.

Define the map $s \colon G ~\ra~ H$ by $s(xb_j) = x$ for $j \in J$ and $x \in H$.
Consider the set
$E = \{ (i \circ s(y), y) ~\big|~ y \in G\}$.
For $j \in J$ and $x \in H$,  $\pi_G(i \circ s(xb_j), xb_j) =  b_j^{-1} x^{-1} x = b_j^{-1}$.
Hence $\pi_G(E) \subset B^{-1} \in \cF$ and so $E \in \cE_{\cF}$ which shows that $i \circ s$ is close to the identity map of $G$.
By Lemma  \ref{coarse_embedding_lemma},  the map $i$ is a coarse equivalence.
\end{proof}

In order apply the above results in the case of the group-compact coarse structure on a topological group we will need to consider the
following hypothesis on a closed subgroup.

\begin{definition}[Property (K)]
\label{Property_K}
A  map $f \colon X ~\ra~ Y$ between Hausdorff spaces has {\it Property (K)}  if for every compact $K \subset Y$ there exists
a compact $K' \subset X$ such that  $f(K') = K$.
Let $G$ be a Hausdorff topological group and $H$ a closed subgroup. 
We say that the pair $(G, H)$ has {\it Property (K)} if the quotient map
$p_H \colon G ~\ra~ G/H$ from $G$ to the space $G/H$ of left cosets of $G$ has Property (K).
\end{definition}

Let $G$ be a topological group and $H$ a subgroup of $G$. 
The subgroup $H$ is said to admit a {\it local cross-section} if there exists a non-empty open subset $U$ of $G/H$ and
a continuous map $s \colon U ~\ra~ G$ such that $p_H \circ s$ is the identity map of $U$.
A local cross-section exists if and only if $p_H$ is a locally trivial $H$-bundle.

\begin{proposition}
\label{local_section}
Let $G$ be a Hausdorff topological group and $H$ a closed subgroup of $G$. 
If $H$ admits a local cross-section then $(G,H)$ has Property (K).
\end{proposition}

\begin{proof}
Let  $s \colon U ~\ra~ G$ be a local cross-section.
The space $G/H$ is a regular topological space (\cite[Theorem 1.5.6]{Arhangel}) and so there exists a non-empty open set $V$
such that $\overline{V} \subset U$.
Since $K$ is compact, it is covered by finitely many translates of $V$, say $K \subset \bigcup^n_{i=1} g_i V$.
Let $K' =  \bigcup^n_{i=1} g_i s(\overline{V_i}  \cap g_i^{-1}K)$.  Then $K'$ is compact and $p_H(K') = K$.
\end{proof}

Property (K) for locally compact subgroups is a consequence of a result of Antonyan (\cite{Antonyan}).

\begin{proposition}
\label{locally_compact_Property_K}
Let $G$ be a Hausdorff topological group and $H$ a closed subgroup of $G$. 
If  $H$ is locally compact then $(G,H)$ has Property (K).
\end{proposition}

\begin{proof}
By \cite[Corollary 1.3]{Antonyan}, there exists a closed subspace $F \subset G$ such that the restriction $(p_H)|_F \colon F ~\ra~ G/H$
is a surjective perfect map.  Hence if $K \subset G/H$ is compact then  $K' = F \cap p_H^{-1}(K)$ is a compact subset of $G$ such that $p_H(K')=K$.
\end{proof}

\begin{proposition}
\label{coarse_equiv_cocompact_subgroup}
Let $G$ be a Hausdorff topological group and let $H$ be a closed subgroup of $G$ such that $G/H$ is compact.
Assume that $(G, H)$ has Property (K).
Then the inclusion map $i \colon (H, \cE_{\cC(H)}) ~\ra~ (G, \cE_{\cC(G)})$ is a coarse equivalence.
\end{proposition}

\begin{proof}
Property (K) for $H$ implies there exists a compact set $B \subset G$ such that $p_H(B) = G/H$ and so
$G = BH$, equivalently, $G = H B^{-1}$.
Note that since $H$ is closed in $G$ we have $\cC(G)|_H = \cC(H)$.
The conclusion of the Proposition  follows from Proposition \ref{coarse_equiv_two}.
\end{proof}

\begin{proposition}
\label{coarse_equiv_compact_normal_subgroup}
Let $G$ be a Hausdorff topological group and let $N$ be a compact normal subgroup of $G$.
Then the quotient map
 $p_N \colon (G, \cE_{\cC(G)}) ~\ra~ (G/N, \cE_{\cC(G/N)})$ is a coarse equivalence.
\end{proposition}

\begin{proof}
Clearly $p_H( \cC(G) ) \subset \cC(G/N)$.
By Proposition \ref{locally_compact_Property_K}, $\cC(G/N)  \subset p_H( \cC(G) )$.
Hence $p_H( \cC(G)  ) = \cC(G/N)$
and so the conclusion of the Proposition   follows from Corollary \ref{coarse_equiv_cor}.
\end{proof}

\begin{remark} 
\label{not_necessarily_normal}  
The assumption in Proposition \ref{coarse_equiv_compact_normal_subgroup} that the subgroup $N$ of $G$ is normal can be eliminated if we interpret  the homogeneous space $G/N$ as a coarse space with the group-compact coarse structure as described in
Remark  \ref{group_compact_generalized}(3).
\end{remark}

%%%%%%%%%%%%%%%%%%%%%%%%%%%%%%%%%%%%%%%%%%%%
%%
%%  Section: Asymptotic Dimension
%%
%%%%%%%%%%%%%%%%%%%%%%%%%%%%%%%%%%%%%%%%%%%%

\section{Asymptotic Dimension}
\label{sec:asymptotic_dimension}

In this section we develop asymptotic dimension theory for a group $G$ with a compatible coarse structure $\cE_{\cF}$.
We give three equivalent characterizations of the assertion $\asdim(G, \cE_{\cF}) \leq n$ (Proposition \ref{asdim_characterizations}).
The other main results are: Theorem \ref{asdim_zero} characterizing groups with $\asdim(G, \cE_{\cF}) =0$, subgroup theorems (Theorems \ref{asdim_of_a_subgroup} and \ref{asdim_by_subgroups}) and an extension theorem (Theorem \ref{extension_theorem});  
in the special case of the group-compact coarse structure on a Hausdorff topological group the corresponding results are, respectively,
Corollaries 
 \ref{asdim_zero_group_compact},
 \ref{asdim_of_a_closed_subgroup},
 \ref{asdim_by_subgroups_group_compact} and
 Theorem \ref{extension_of_top_groups}.

Let $(X, \cE)$ be a coarse space
and  let $\cU$ be a collection of subsets of $X$.
Let $L \in \cE$ be an entourage.
The collection $\cU$  is said to be {\it $L$-disjoint} 
if for all $A, B \in \cU$ such that $A \neq B$  the sets $A \times B$ and $L$ are disjoint.
A {\it uniform bound} for $\cU$ is an entourage $M \in \cE$ such that
$A \times A \subset M$ for all $A \in \cU$.
The collection $\cU$ is {\it uniformly bounded} if a uniform bound for $\cU$ exists.

\begin{definition}
\label{asdim_def}
Let $(X, \cE)$ be a coarse space and $n$ a non-negative integer.
Then $\asdim(X,\cE) \leq n$ if for every entourage $L \in \cE$ there is a cover
$\cU$ of $X$ such that:
\begin{enumerate}

\item  $\cU = {\cU}_0 \cup \cdots  \cup \ {\cU}_n$, 

\item  ${\cU}_i$ is $L$-disjoint for each index $i$,

\item  $\cU$ is uniformly bounded.

\end{enumerate}

If no such integer exists, we say $\asdim(X,\cE) = \infty $.
If $\asdim(X,\cE) \leq n$  and $\asdim(X,\cE) \leq n-1$ is false then we say  $\asdim(X,\cE) = n$
and the integer $n$ is called the {\it asymptotic dimension of $X$ (with respect to $\cE$)}.

\end{definition}

Definition \ref{asdim_def} differs slightly from  Roe's original definition  (\cite[Definition 9.1]{Roe}) in that he assumes $\cU$ is countable.
Grave gives the following equivalent characterization of the assertion $\asdim(X, \cE) \leq n$.

\begin{theorem} (\cite[Theorem 9]{Grave})
\label{theorem_of_Grave}
Let $(X, \cE)$ be a coarse space and $n$ a non-negative integer.
Then $\asdim(X,\cE) \leq n$ if and only if for every entourage $L \in \cE$ there is a cover
$\cU$ of $X$ such that:
\begin{enumerate}

\item  The multiplicity of $\cU$ is less than or equal to $n+1$ (that is, every point of $X$ is contained in at most $n+1$ elements of $\cU$),

\item  for all $x \in X$ there exists $U \in \cU$ such that $L(x) \subset U$,

\item  $\cU$ is uniformly bounded.

\end{enumerate}

\end{theorem}

Let $G$  be a group and let $A, B, K$ be subsets of $G$.
We say that $A$ and $B$ are {\it $K$-disjoint} if $(B^{-1}A) \cap K = \emptyset$.
We say that a collection $\cP$ of subsets of $G$ is  $K$-disjoint if 
for  every $A,~B \in \cP$  such that $A \neq B$ the sets $A$ and $B$ are $K$-disjoint.

In the context of groups with the compatible coarse structures, Definition \ref{asdim_def} can be  reformulated as follows.

\begin{proposition}
\label{prop_asdim_groups}
Let $G$ be a group and $\cF$ a generating family.
If for every $K \in \cF$ there is a cover
$\cU$ of $G$ such that:
\begin{enumerate}

\item  $\cU = {\cU}_0 \cup \cdots  \cup \ {\cU}_n$, 

\item  $\cU_i$ is $K$-disjoint for each index $i$,

\item  $\cU$ is uniformly bounded  (see Remark {\rm \ref{remark_asdim_groups}}),
\end{enumerate}
then $\asdim(G,\cE_{\cF}) \leq n$.
Conversely, if $(G, \cE_{\cF})$ is connected and  $\asdim(G,\cE_{\cF}) \leq n$ then for every $K \in \cF$ there is a cover
$\cU$ of $G$ satisfying Conditions {\rm (1)},  {\rm (2)} and  {\rm (3)}.
\end{proposition}

\begin{proof}
Let $L \in \cE_{\cF}$. Then $L \subset G(K' \times K')$ for some $K' \in \cF$.    Let $K = (K')^{-1}K'$.   Note that $K \in \cF$.
Let $\cU$ be a collection of subsets of $G$ satisfying Conditions (1), (2) and (3) for $K$.
Let $A,~B \in \cU_i$ with $A \neq B$. Then
\[
\pi_G((A \times B) \cap L) \subset  \pi_G((A \times B) \cap G(K' \times K') ) = (B^{-1}A) \cap ((K')^{-1}K )= (B^{-1}A) \cap K = \emptyset.
\]
Hence $(A \times B) \cap L = \emptyset$ and so the $\cU_i$'s are $L$-disjoint.  
Hence $\asdim(G,\cE_{\cF}) \leq n$.

For the converse, assume that $(G, \cE_{\cF})$ is connected and  $\asdim(G,\cE_{\cF}) \leq n$.
Let $ K \in \cF$.
By Proposition \ref{prop_connected}, $\{1\} \in \widehat{\cF}$ and so $K \cup \{1\} \subset \widetilde{K}$ for some  $\widetilde{K} \in \cF$.
Let $L = G( \widetilde{K} \times  \widetilde{K})$.
Note that $ L \in \cE_{\cF}$.
Let $\cU$ be a family of subsets of $G$ satisfying Conditions (1), (2) and (3) in Definition \ref{asdim_def}.
For each index $i$ and for every $A,~B \in \cU_i$ with $A \neq B$ we have  $(A \times B) \cap L = \emptyset$.  
Hence $(B^{-1}A) \cap (\widetilde{K})^{-1}\widetilde{K} = \pi_G(( A \times B) \cap L )= \emptyset$.  
Since $K \subset (\widetilde{K} )^{-1}\widetilde{K}$, it follows that $(B^{-1}A) \cap K =\emptyset$. 
\end{proof}

\begin{remark}
\label{remark_asdim_groups}
In Proposition \ref{prop_asdim_groups},  Condition (3)  (that $\cU$ is uniformly bounded) implies the condition:
\begin{enumerate}
\item[($3'$)]   There exists $F \in \cF$ such that for all  $A \in \cU$, \,$A^{-1}A \subset F$.
\end{enumerate}
If $(G, \cE_{\cF})$ is connected then Condition ($3'$) implies Condition (3).
\end{remark}

\begin{proposition}
\label{asdim_characterizations}
Let $G$ be a group and $\cF$ a generating family on $G$.  Assume $(G, \cE_{\cF})$ is connected.
Let $n$ be a non-negative integer.
The statements (A), (B) and (C) below are all equivalent to the assertion that $\asdim(G, \cE_{\cF}) \leq n$.
\end{proposition}
\begin{enumerate}

\item[(A)]  For all $K \in \cF$ there exists a cover $\cU$ of $G$ such that
\begin{enumerate}
\item[(1)]     $\cU = {\cU}_0 \cup \cdots  \cup \ {\cU}_n$, 

\item[(2)]     $\cU_i$ is $K$-disjoint for each index $i$,

\item[(3)]    $\cU$ is uniformly bounded.
\end{enumerate}

\item[(B)]  For all $K \in \cF$ there exists a cover $\cV$ of $G$ such that:
\begin{enumerate}
\item[(1)]     For each $g \in G$ at most $n+1$ elements of $\cV$ meet $gK$,

\item[(2)]    $\cV$ is uniformly bounded.
\end{enumerate}

\item[(C)]  For all $K \in \cF$ there exists a cover $\cW$ of $G$ such that
\begin{enumerate}
\item[(1)]     $\cW$ has multiplicity less than or equal to $n+1$,

\item[(2)]    For all $g \in G$ there exists $W \in \cW$ such that $gK \in W$,

\item[(3)]    $\cW$ is uniformly bounded.
\end{enumerate}

\end{enumerate}

\begin{proof}    The equivalence of (A) and the assertion $\asdim(G, \cE_{\cF}) \leq n$ is Proposition \ref{prop_asdim_groups}.

\smallskip

\noindent
{\it Proof that (A) implies (B)}.
Let $K \in \cF$ be given.  Let ${\widetilde K} = K^{-1} K$.
Assuming (A), there exists a cover $\cU = {\cU}_0 \cup \cdots  \cup \ {\cU}_n$ of $G$ which is uniformly bounded and such that
each ${\cU}_i$ is ${\widetilde K}$-disjoint.
Let $g \in G$ and let $U_1, \, U_2 \in \cU$ be such that $U_1 \neq U_2$ and $(gK) \cap U_1 \neq \emptyset$ and  $(gK) \cap U_2 \neq \emptyset$.
Then $(U_2^{-1}U_1)  \cap  {\widetilde K} \neq \emptyset$. 
Hence $U_1$ and $U_2$ are not $K$-disjoint and so cannot belong to the same ${\cU}_i$.
It follows that at most $n+1$ elements of $\cU$ meet $gK$.

\smallskip

\noindent
{\it Proof that (B) implies (C)}.   
Let $K \in \cF$ be given.  Assuming (B),  there exists a uniformly bounded cover $\cV$ such that for any $g \in G$ at most $n+1$ elements of $\cV$ meet $gK^{-1}$.
Define $\cW = \{ VK ~\big|~ V \in \cV \}$.
Clearly, $\cW$ is a uniformly bounded cover of $G$.
Let $g \in G$.  Then $g \in VK$ for some $V \in \cV$ and there exists $h \in V$ such that $g \in hK$.  It follows that $h \in V \cap (gK^{-1})$ and thus
$V \cap (gK^{-1}) \neq \emptyset$.
Since there are at most $n+1$ sets $V \in \cV$ which meet $gK^{-1}$, it follows that there are at most $n+1$ sets $VK \in \cW$  containing $g$, 
that is, the multiplicity of $\cW$ is less than or equal to $n+1$.
Since $\cV$ covers $G$, any set of the form $gK$ is contained in some $VK \in  \cW$.

\smallskip

\noindent
{\it Proof that (C) implies (A)}.   Let $L \in \cE_{\cF}$ be given.  We may assume $L$ is of the form $L = G(K' \times K')$ where $K' \in \cF$.
Let $K=( K')^{-1} K'  \in \cF$.  
Assuming (C), there exists a uniformly bounded cover $\cW$ of $G$ with multiplicity less than or equal to $n+1$ such that for every $g \in G$
there exists $W \in \cW$ such that $L(g) = g( K')^{-1} K'   = gK \subset W$. Statement (A) now follows from Theorem \ref{theorem_of_Grave} and 
Proposition \ref{prop_asdim_groups}.
\end{proof}

We have the following characterization of groups with asymptotic dimension zero with respect to a compatible coarse structure.

\begin{theorem}
\label{asdim_zero}
Let $G$ be a group and $\cF$ a generating family on $G$.
Assume that $(G, \cE_{\cF})$ is connected
and that  $G$ has a set, $S$,  of generators (as an abstract group) such that $S \in \widehat{\cF}$.
Then $\asdim(G, \cE_{\cF}) =0$ if and only if $G \in \cF$.
\end{theorem}

\begin{proof}
If $G \in \cF$ then the singleton set $\cU_0 = \{G\}$ is a uniformly bounded cover of $G$ which is vacuously $L$-disjoint
for any $L \in \cE_{\cF}$ and so $\asdim(G, \cE_{\cF}) =0$.

Assume  that $\asdim(G, \cE_{\cF}) =0$ and that $S \in  \widehat{\cF}$ is a set of generators.  Let $K \in \cF$ be such that $S \subset K$.
Since $(G, \cE_{\cF})$ is connected, we can find such a set $K$ so that $1 \in K$.
Furthermore, by replacing $K$ with $K \cup K^{-1} \in \cF$ we may assume $K$ is symmetric.
By Proposition \ref{prop_asdim_groups},  
there is a uniformly bounded cover, $\cU_0$, of $G$ such that for all $A, B \in \cU_0 $ with $A \neq B$ we have $(B^{-1}A )\cap K = \emptyset$.
This condition on $A$ and $B$ implies  $A \cap (BK) = \emptyset$ and thus also $A \cap B = \emptyset$ because $1 \in K$.
Let $B \in \cU_0$ and let $X$ be the union of all elements of $\cU_0$ other than $B$.
Since $\cU_0$ is a cover of $G$, we have $G = X \cup B$.   
Also $X \cap B = \emptyset$ and $X \cap(BK) = \emptyset$.
Hence  $B = BK$.  It follows that   $B K^n = B$ for any positive integer $n$ (where $K^n = K \cdots K$,  $n$ factors).
Since $K$ is a symmetric set of generators for $G$ and contains $1 \in G$,  we have 
$G = \bigcup_{n \geq 1}  K^n$ and so  $B = G$.
It follows that $\cU_0 = \{G\}$.  
Since $\cU_0$ is uniformly bounded, $G \in \cF$.
\end{proof}

\begin{definition}
The {\it asymptotic dimension} of a Hausdorff topological group $G$, denoted $\asdim(G)$,  is $\asdim(G, \cE_{\cC(G)} )$ where $\cE_{\cC(G)}$ is the group-compact
coarse structure on $G$ as in Example \ref{EX_group_compact}.
\end{definition}

An isomorphism of topological groups is clearly a coarse equivalence with respect to their group-compact coarse structures and hence preserves asymptotic dimension
(since, in general, coarse equivalences of coarse spaces preserve asymptotic dimension).  Also for any Hausdorff topological group $G$, if  $G_{\rm k}$ is the Hausdorff topological group obtained by re-topologizing $G$ with the weak topology determined by its collection of compact subsets then
$\asdim(G) = \asdim(G_{\rm k})$ because $G$ and $G_{\rm k}$ have the same compact subsets (see Remark \ref{group_compact_generalized}(2)).

Theorem  \ref{asdim_zero} immediately  yields:

\begin{corollary} 
\label{asdim_zero_group_compact} 
Let $G$ be a Hausdorff topological group with a compact set of generators.
Then  $\asdim(G) = 0$ if and if only if $G$ is compact.
\end{corollary}

\begin{example}
\label{EX_roots_of_unity}
Let $C$  be the  topological  subgroup of $S^1$ (complex numbers of unit modulus) given by
$C= \{  e^{2\pi i \, m/2^n}  ~\big|~  m, n  \in \Z \}$. 
The set $ \{  e^{2\pi i /2^n}  ~\big|~ n = 0, 1, \dots \}  \subset C$ is a compact set of generators for $C$.
Since $C$ is not compact,  Corollary \ref{asdim_zero_group_compact}  implies $\asdim(C) >0$.
\end{example}

\begin{theorem}
\label{asdim_of_a_subgroup}
Let $G$ be a group and $\cF$ a generating family on $G$. Assume that $(G, \cE_{\cF})$ is connected.
Let $H \subset G$ be a subgroup.
Then $\asdim(H, \cE_{\cF|_H})  \leq \asdim(G, \cE_{\cF})$ .
\end{theorem}

\begin{proof}
If $\asdim(G, \cE_{\cF}) = \infty$ then there is nothing to prove so assume $\asdim(G, \cE_{\cF}) = n$  where $n$ is finite.
Let  ${\widetilde K} \in \cF|_H$.  Then  ${\widetilde K}=  K \cap H$ for some $K \in \cF$.
By Proposition \ref{prop_asdim_groups} there is a cover 
$\cU = {\cU}_0 \cup \cdots  \cup \ {\cU}_n$ of $G$ satisfying Conditions (1), (2) and ($3'$) in that Proposition
(see Remark  \ref{remark_asdim_groups} for ($3'$)).
Let  ${\cU}_i |_H  = \{ U \cap H ~\big|~ U \in {\cU}_i\}$ for $i=0, \ldots, n$.
Then $\cU |_H  = {\cU}_0 |_H \cup \cdots  \cup \ {\cU}_n |_H$ is a cover of $H$ satisfying Conditions (1), (2) and ($3'$)
for ${\widetilde K}$.
Hence, by  Proposition \ref{prop_asdim_groups},  $\asdim(H, \cE_{\cF|_H})  \leq n$.
\end{proof}

If $G$ is a Hausdorff topological group and $H$ is a  closed subgroup then $\cC(G)|_H = \cC(H)$ and so 
Proposition \ref{asdim_of_a_subgroup} yields the following corollary.

\begin{corollary} 
\label{asdim_of_a_closed_subgroup} 
Let $G$ be a Hausdorff topological group  and $H$ a closed subgroup of $G$.
Then $\asdim(H)  \leq \asdim(G).$ \qed
\end{corollary}

\begin{example}[Virtually connected Lie groups]
\label{Lie_groups}
Let $G$ be a virtually connected Lie group, and let $K$ be a maximal compact subgroup of $G$.
If $\Gamma$ is a discrete subgroup of $G$, then $G/K$ is a finite dimensional $\Gamma$-CW complex and is a model for the universal proper $\Gamma$-space $\Ebar\Gamma$~\cite{Lueck}.
By~\cite[Section 3]{Carlsson_Goldfarb} and \cite[Proposition 3.3]{Ji}, $\asdim(G/K, \cE_d)=\dim(G/K)$, where $d$ is any $G$-invariant Riemannian metric on $G/K$.
Recall from Remark~\ref{group_compact_generalized}(3) that $G/K$  carries the group-compact coarse structure  $\cE_{\text{$G$-cpt}}$ which is completely 
determined by the $G$-action on $G/K$ and the compact subsets of $G/K$.
By an extension of Theorem~\ref{group-compact_equals_metric} to homogeneous spaces of the form $G/K$ with $K$ compact,
$\asdim(G/K,  \cE_{\text{$G$-cpt}})=\asdim(G/K, \cE_d)$.
By Remark~\ref{not_necessarily_normal}, $\asdim(G)=\asdim(G/K,  \cE_{\text{$G$-cpt}})=\dim(G/K)$.
Furthermore, by Corollary \ref{asdim_of_a_closed_subgroup}, 
the discrete group $\Gamma$ has finite asymptotic dimension less than or equal to $\dim(G/K)$.
\end{example}

The asymptotic dimension of a group with respect to a given compatible coarse structure  is determined by the asymptotic dimensions of a sufficiently large family of subgroups as follows.

\begin{theorem}
\label{asdim_by_subgroups}
Let $G$ be a group and $\cF$ a generating family.
Assume that $(G, \cE_{\cF})$ is connected.
Let $\cH$ be a collection of subgroups of $G$ with the property that for every $K \in \cF$ there exists an $H \in \cH$
such that $K \subset H$.
Then
\[
 \asdim(G, \cE_{\cF} )  ~=~  \sup \left\{  \asdim\left( H, \cE_{\cF|_H} \right)   ~\big|~  H \in \cH \right\}.
 \]
\end{theorem}

\begin{proof}
Our method of proof is motivated by the proof of \cite[Theorem 2.1]{Dranish_and_Smith}.

By Proposition \ref{asdim_of_a_subgroup},
$\sup \left\{  \asdim\left( H, \cE_{\cF|_H} \right)   ~\big|~  H \in \cH \right\} \leq  \asdim(G, \cE_{\cF} )$.
If the left side of this inequality is infinite then there is nothing more to prove so we may assume that
$n = \sup \left\{  \asdim\left( H, \cE_{\cF|_H} \right)   ~\big|~  H \in \cH \right\}$ is finite.

Let $K \in \cF$.  By hypothesis, there exists $H \in \cH$ such that $K \subset H$.
Since $ \asdim\left( H, \cE_{\cF|_H} \right)   \leq n$, 
by Proposition \ref{prop_asdim_groups}
there exists a cover  $\cU = {\cU}_0 \cup \cdots  \cup \ {\cU}_n$ of $H$  such that for some
$L \in \cF$  we have $U^{-1} U \subset L \cap H$ for all $U \in \cU$ and each $\cU_i$ is $K$-disjoint, that is,
$A, B \in \cU_i$ and $A \neq B$ implies $(B^{-1}A) \cap K = \emptyset$.

Let $Z$ be a set of left coset representatives of $H$ in $G$ (hence $G$ is a disjoint union of  the sets $gH$,  $g \in Z$).
Define
$
\cV_i  ~=~ \{ gU ~\big|~ U \in \cU_i  \text{ and } g \in Z \}
$
for $ i=0,\ldots,n$.
Clearly,  $\cV = {\cV}_0 \cup \cdots  \cup \ {\cV}_n$ is a cover of $G$. 
Since for $g \in Z$ and $U \in \cU$, we have $(gU)^{-1}(gU) = U^{-1}U \subset L \cap H \subset L$ and so $\cV$ is uniformly bounded.

We claim that each $\cV_i$ is $K$-disjoint.   For a given $i$, let $gU,  g'U' \in \cV_i$ where $U, U' \in \cU_i$ and $g, g' \in Z$.
Assume that $gU \neq g'U'$.
If $g=g'$ then $U \neq U'$ and so  $((g'U')^{-1}(gU)) \cap K =  ((U')^{-1}U) \cap K = \emptyset$.
If $g \neq g'$ then  $(gH) \cap (g'H) = \emptyset$ because $g, g'$ are representatives of distinct left cosets.
Since $K \subset H$, we have
\[
((g'U')^{-1}(gU)) \cap K \subset ((g'H)^{-1}(gH)) \cap H = \emptyset.
\]
Hence $\cV_i$ is $K$-disjoint. 
By Proposition \ref{prop_asdim_groups},  $\asdim(G, \cF) \leq n$ and so equality holds since the opposite inequality was previously established.
\end{proof}

Given a group $G$ and a subset $S \subset G$,  let  $\langle S \rangle$ denote the subgroup of  $G$ generated by $S$.

\begin{corollary}
\label{asdim_by_subgroups_group_compact}
Let $G$ be a Hausdorff topological group.
Then
\[
 \asdim(G)  ~=~  \sup \left\{  \asdim\left(\overline{\langle K \rangle} \right) ~\big|~  K \in \cC(G)  \right\}
 \]
where $\overline{\langle K \rangle}$ is the closure of ${\langle K \rangle}$ in $G$.
\end{corollary}

\begin{proof}
The collection $\cH = \big\{  \overline{\langle K \rangle} ~|~  K \in \cC(G) \big\}$ clearly satisfies the hypothesis of Theorem \ref{asdim_by_subgroups} and furthermore
$\cC(G)|_{\overline{\langle K \rangle} } = \cC\left(\overline{\langle K \rangle} \right)$ since $\overline{\langle K \rangle}$ is a closed subgroup of $G$.
\end{proof}

We have the following estimate for the asymptotic dimension of an extension.

\begin{theorem}[Extension Theorem]
\label{extension_theorem}
Let $1\to N\xrightarrow{i}  G\xrightarrow{\pi} Q\to 1$ be an extension of groups.   Let $\cF$ be a generating family on $G$.
Assume that $(G, \cE_{\cF})$ is connected.
If $\asdim(i(N),\cE_{{\cF}|_{i(N)}})\leq n$ and $\asdim(Q, \cE_{\pi(\cF)})\leq k$  then $\asdim(G, \cE_{\cF})\leq (n+1)(k+1)-1$.
\end{theorem}

\begin{proof}
We may assume that $i \colon N ~\ra~G$ is the inclusion of a subgroup.
By Proposition \ref{coarse_equiv_one},
$\phi \colon (G, \cE_{N\cF}) ~\ra~ (Q, \cE_{\phi(\cF)})$ is a coarse equivalence and hence
$\asdim(G, \cE_{N\cF}) = \asdim(Q, \cE_{\pi(\cF)})$.

Let $K \in \cF$ be given.   Since $\asdim(G, \cE_{N\cF})  \leq k$, 
by Proposition \ref{prop_asdim_groups} there exists  $K' \in \cF$  and a cover
 $\cU = {\cU}_0 \cup \cdots  \cup \ {\cU}_k$ of $G$ such that each ${\cU}_i$ is $NK$-disjoint and 
 $U^{-1}U \subset K'N$ for all $U \in \cU$.
 
Since $\asdim(N, \cE_{{\cF}|_N}) \leq n$, 
there exists  $K'' \in \cF$  and a cover
 $\cV = {\cV}_0 \cup \cdots  \cup \ {\cV}_n$ of $N$ such that each ${\cV}_j$ is $K' K (K')^{-1}$-disjoint and 
 $V^{-1}V \subset K'' $ for all $V \in \cV$.

For each $U \in \cU$, choose an element $g_U\in U$.
Given $0\leq i\leq k$ and $0\leq j\leq n$, define $\cW_{ij}=\{(g_UVK')\cap U ~\big|~ U\in \cU_i, V\in \cV_j\}$. 
We will show that $\cW=\bigcup_{i,j}\cW_{ij}$ is a uniformly bounded cover of $G$ such that each $\cW_{ij}$ is $K$-disjoint. 
Let $g \in G$. Then $g \in U$ for some $U \in {\cU}_i$.
We have that $g^{-1}_U g \in U^{-1} U \subset NK'$ and so  $g^{-1}_U g \in VK'$ for some $V \in {\cV}_j$ because $\cV$ covers $N$.
Hence $g = g_U ( g^{-1}_U g) \in (g_U V K') \cap U \in {\cW}_{ij}$ which shows that $\cW$ covers $G$.
If  $W = (g_U VK') \cap U \in {\cW}_{ij}$ then
\[
W^{-1}W \subset (g_U VK')^{-1}(g_U VK') = (K')^{-1}V^{-1}V K'  \subset (K')^{-1}K'' K' \in \cF
\]
and so $\cW$ is uniformly bounded.

Let $A,B\in \cW_{ij}$ with $A\neq B$.
Write $A=(g_{U_A}V_A K')\cap  U_A$ and $B=(g_{U_B}V_B K') \cap  U_B$, where $U_A,U_B\in \cU_i$ and $V_A,V_B\in \cV_j$. 
If $U_A \neq U_B$ then $(U_B^{-1} U_A) \cap (NK) = \emptyset$ because $\cU_i$ is $NK$-disjoint.
Since $(B^{-1}A) \cap K \subset  (U_B^{-1} U_A) \cap (NK)$ it follows that  $(B^{-1}A) \cap K = \emptyset$.
If $U_A= U_B$ then $g_{U_A}=g_{U_B}$ and  $\left( B^{-1}A \right)\cap K\subset \left( (K')^{-1}V_B^{-1}V_A K' \right)\cap K$.
The right side of this inclusion is empty because ${\cV}_j$ is $K' K (K')^{-1}$-disjoint and so $\left(B^{-1}A \right)\cap K = \emptyset$.
Hence  $\cW_{ij}$  is $K$-disjoint.
\end{proof}

\begin{theorem}
\label{extension_of_top_groups}
Let $1\to N\xrightarrow{i}  G\xrightarrow{\pi} Q\to 1$ be an extension of Hausdorff topological groups.  Assume  that $i$ is a homeomorphism onto its image and that
$\pi$ has Property (K)  (see Definition {\rm \ref{Property_K}}).
If $\asdim(N)\leq n$ and $\asdim(Q)\leq k$  then $\asdim(G)\leq (n+1)(k+1)-1$.
\end{theorem}

\begin{proof}
By hypothesis, $\pi$ is continuous and so $i(N)=\ker(\pi)$ is a closed subgroup of $G$ and thus $\cC(i(N)) = \cC(G)|_{i(N)}$. 
Also, because $i\colon N ~\ra~ i(N)$ is an isomorphism of topological groups it is a coarse equivalence (with the group-compact coarse structures) and so
$\asdim(N) = \asdim(i(N))$.
Clearly,  $\pi(\cC(G)) \subset \cC(Q)$.
If $\pi$ has Property (K)  then $\cC(Q) \subset \pi(\cC(G))$ and so 
$\cC(Q) = \pi(\cC(G))$.
The conclusion of the Theorem follows from Theorem \ref{extension_theorem}.
\end{proof}

We give some sufficient conditions for the map $\pi$ in Theorem \ref{extension_of_top_groups} to have Property (K).

\begin{proposition} Let $\pi \colon G ~\ra~ Q$ be a continuous, surjective homomorphism of Hausdorff topological groups.  
Assume that $\pi$ has one of the following properties.
\begin{enumerate}
\item  $\pi$ admits a local cross-section, that is, there exists a non-empty open set $U \subset Q$ and
a continuous map $s \colon U ~\ra~ G$ such that $\pi \circ s$ is the identity map of $U$.

\item  $\pi$ is an open map and $\ker(\pi)$ is locally compact.
\end{enumerate}
Then $\pi$ has Property (K).
\end{proposition}

\begin{proof}
If $\pi$ admits a local cross-section then a
straightforward modification of Proposition \ref{local_section} gives that $\pi$ has Property (K).

If $\pi$ is open then it  factors as $\pi = {\bar \pi}  \circ p_{\ker(\pi)}$, where
${\bar \pi} \colon G/\ker(\pi) ~\ra~ Q$ is a homeomorphism.
If, in addition,  $\ker(\pi)$ is locally compact then it follows from Proposition \ref{locally_compact_Property_K} that
$\pi$ has Property (K).
\end{proof}

%%%%%%%%%%%%%%%%%%%%%%%%%%%%%%%%%%%%%%%%%%%%
%%
%%  Section:  The Asymptotic Dimension of the Free Topological Group
%%
%%%%%%%%%%%%%%%%%%%%%%%%%%%%%%%%%%%%%%%%%%%%

\section{The Asymptotic Dimension of the Free Topological Group}

Given a topological space $X$,  a {\it  free topological group on $X$}  is a pair  $(F_{\rm top}(X), i)$ consisting of a Hausdorff topological group $F_{\rm top}(X)$ together with a continuous map
$i \colon X ~\ra~ F_{\rm top}(X)$ satisfying the following universal property:
For every continuous map $f \colon X ~\ra~ H$ to an arbitrary Hausdorff topological group $H$ there exists a unique continuous homomorphism $F \colon F_{\rm top}(X) ~\ra~ H$ such
that $f = F \circ i$.  A standard argument of category theory shows that if $(F_{\rm top}(X), i)$ exists then it is unique up to a unique isomorphism,
that is, if  $(F'_{\rm top}(X), i')$ also satisfies the defining universal
property then there exists a unique isomorphism of topological groups $\Phi \colon F_{\rm top}(X) ~\ra~   F'_{\rm top}(X)$ such that $\Phi \circ i = i'$.

Markov  (\cite{Markov})  proved that a free topological group,  $(F_{\rm top}(X), i)$,  on a Tychonoff  (``completely regular'')  space $X$ exists and  that $i \colon X ~\ra~ F_{\rm top}(X)$ is a topological
embedding and $i(X)$ algebraically generates  $F_{\rm top}(X)$.  Furthermore,  the discrete group obtained by forgetting the topology on $F_{\rm top}(X)$  is  algebraically a free group generated by $i(X)$.
See \cite[Chapter 7]{Arhangel} for a contemporary exposition of the theory of free topological groups.
 
Henceforth, we will identify $X$ with its image $i(X)$ in $F_{\rm top}(X)$.  
 
 \begin{proposition}
\label{group-compact_equals_bounded_for_freetop}
If $X$ is a compact Hausdorff space then the group-compact coarse structure on $F_{\rm top}(X)$ coincides with the bounded coarse structure associated to the word metric, $d_X$,
determined by the generating set $X \subset F_{\rm top}(X)$.  
\end{proposition}

\begin{proof}
The $d_X$-ball of non-negative integer radius $n$ is $(X \cup \{1\} \cup X^{-1})^n$, which is compact since $X$ is compact.
Hence Condition (ii) of Theorem \ref{group-compact_equals_metric} holds.

Let $K \subset F_{\rm top}(X)$ be compact.
By \cite[Corollary 7.4.4]{Arhangel},  $K \subset  (X \cup \{1\} \cup X^{-1})^n$ for some $n$.
Hence Condition (i) of Theorem \ref{group-compact_equals_metric} holds.
\end{proof}

\begin{corollary}
\label{asdim_of_freetop_on_compact}
If $X$ is a non-empty compact Hausdorff space then $\asdim(F_{\rm top}(X)) =  1$. 
\end{corollary}

\begin{proof}
Let  $F(X)$ denote the free group generated by $X$ (forgetting its topology).
By Proposition \ref{group-compact_equals_bounded_for_freetop}, 
$\asdim(F_{\rm top}(X))  = \asdim(F(X), \cE_{d_X})$, where $ \cE_{d_X}$ is the bounded coarse structure associated to the word metric $d_X$.
Let $T$ be the Cayley graph of  $F(X)$ with respect to the set of generators $X \subset F(X)$ and let $d_T$ be the natural distance on  $T$.
Note that $T$ is a tree because any nontrivial loop in $T$ would give rise to a non-trivial relation in $F(X)$.
Since $(F(X), d_X)$ is quasi-isometric to  $(T, d_T)$, they have the same asymptotic dimension (with respect to the bounded coarse structures determined by the given metrics).
Any metric tree has asymptotic dimension at most $1$ (\cite[Example, \S3.1]{Bell_and_Dranish}).  
Also, since $X$ is non-empty, $T$ is an unbounded tree and thus has positive asymptotic dimension.
Hence $ \asdim(T, \cE_{d_T})=  1$  and so 
\[
\asdim(F_{\rm top}(X))  = \asdim(F(X), \cE_{d_X}) = \asdim(T, \cE_{d_T})=  1.
\]
\end{proof}

Corollary \ref{asdim_of_freetop_on_compact} can be generalized to a large class of non-compact spaces as follows.

\begin{theorem}
\label{asdim_of _freetop_general}
If $X$ is a non-empty space that is homeomorphic to a closed subspace of a Cartesian product of metrizable spaces
then  $\asdim(F_{\rm top}(X)) =  1$.
\end{theorem}

\begin{proof}
A space $X$  is homeomorphic to a closed subspace of a Cartesian product of metrizable spaces  if and only if it is {\it Dieudonn\'e complete}, that is,
there exists a complete uniformity on $X$ (\cite[8.5.13]{Engelking}).
\smallskip

\noindent{\it Claim.}  For such a space the following holds:
\[
 \asdim(F_{\rm top}(X)) ~=~  \sup \left\{  \asdim\left(F_{\rm top}(A) \right) ~\big|~  A \subset X \text{ is compact}  \right\}.
 \]
Assuming the claim, the conclusion of the Theorem follows from Corollary \ref{asdim_of_freetop_on_compact} because
$\asdim\left(F_{\rm top}(A)  \right) =1$ for $A$ compact.
We now prove the claim.

The {\it support}  of a reduced 
word $g = x_1^{\pm 1} \cdots x_n^{\pm 1}  \in F_{\rm top}(X)$ is, by definition, the set $\supp(g) = \{x_1, \ldots, x_n\} \subset X$. 
The {\it support} of $B  \subset F_{\rm top}(X)$ is the set $\supp(B) = \cup_{g \in B} \supp(g)$.

Let $K \subset F_{\rm top}(X)$ be compact.  
Let $A$ be the closure of $\supp(K)$ in $X$.
Since $X$ is Dieudonn\'e complete,
\cite[Corollary 7.5.6]{Arhangel} implies that  $A$ is compact.
Clearly $K \subset F_{\rm top}(X,A)$, where $F_{\rm top}(X,A)$ denotes the subgroup of $F_{\rm top}(X)$ generated by $A$. 
By \cite[Corollary 7.4.6]{Arhangel}, 
if $A$ is compact then
$F_{\rm top}(X,A)$ is a closed subgroup of $ F_{\rm top}(X)$ and  $F_{\rm top}(X,A)$  is isomorphic to  $F_{\rm top}(A)$ as topological groups.
Hence $\cC(F_{\rm top}(X))|_{F_{\rm top}(X,A)}   =  \cC(F_{\rm top}(X,A))$  and
$
\asdim(F_{\rm top}(X,A) )  = \asdim(F_{\rm top}(A) ).
$
Applying Theorem \ref{asdim_by_subgroups} to the collection
 $\cH =  \{   F_{\rm top}(X,A) ~|~  A \subset X \text{ is compact} \}$
 yields the claim.
\end{proof}

We observe that  $F_{\rm top}(X)$ is typically not locally compact
and so Proposition \ref{locally_compact_groups} does not apply to it.
Combining various results of \cite{Arhangel} and \cite{Morris} yields
the following proposition, presumably well known to experts.

\begin{proposition}
\label{not_locally_compact}
Let $X$ be a locally compact metric space. Then $F_{\rm top}(X)$ is locally compact if and only if $X$ is discrete.
\end{proposition}

\begin{proof}
Clearly, if $X$ is discrete then $F_{\rm top}(X)$ is also discrete and hence also locally compact.

Assume $X$ is not discrete.
By \cite[Theorem 7.1.20]{Arhangel},  $F_{\rm top}(X)$ is not first countable.
By hypothesis, $X$ is a locally compact metric space and so \cite[Corollary 1]{Morris} asserts that
$F_{\rm top}(X)$ has no small subgroups, that is, there is a neighborhood of the identity which contains no subgroups other than the trivial group.
By \cite[Theorem 3.1.21]{Arhangel},  a locally compact group with no small subgroups is first countable. 
In particular,§ $F_{\rm top}(X)$ cannot be locally compact.
\end{proof}

Proposition \ref{not_locally_compact} implies that $F_{\rm top}(X)$ is not locally compact if the Tychonoff space $X$ contains a compact, metrizable, non-discrete subspace
(because if $Y$ is such a subspace of $X$ then, since $Y$ is compact,  \cite[Corollary 7.4.6]{Arhangel} gives that $F_{\rm top}(Y)$ is isomorphic as
a topological group to a closed subgroup of $F_{\rm top}(X)$ and so
$F_{\rm top}(Y)$ would be locally compact if $F_{\rm top}(X)$ was locally compact).

\begin{remark}
Let $G$ be a Hausdorff topological group.  Let $G^\delta$ denote the discrete group with same underlying group as $G$.
In the case $G = F_{\rm top}(X)$ and for $X$ as in Theorem \ref{asdim_of _freetop_general},
$\asdim(G) = 1 = \asdim( G^\delta)~$ (note that $G^\delta$ is algebraically a free group).
By contrast, if $C$ is the topological group of $2$-power roots of unity then
$\asdim(C)  > 0$ (Example \ref{EX_roots_of_unity}) whereas
$ \asdim(C^\delta) =0$ since $C$ is a torsion group.
\end{remark}

%\noindent$\bullet~\bullet~\bullet~\bullet~\bullet~\bullet~\bullet~\bullet~\bullet~\bullet~\bullet~\,\bullet$ 

%\noindent$\bullet$FIX ME:   

%%%%%%%%%%%%%%%%%%%%%%%%%%%%%%%%%%%%%%%%%%%%
%%
%%  References 
%%
%%%%%%%%%%%%%%%%%%%%%%%%%%%%%%%%%%%%%%%%%%%%

%\bibliographystyle{amsplain}
%\bibliography{NR-CSG.bib}

\begin{thebibliography}{10}

\bibitem{Abels}
H.~Abels, \emph{Reductive groups as metric spaces}, Groups: topological,
  combinatorial and arithmetic aspects, London Math. Soc. Lecture Note Ser.,
  vol. 311, Cambridge Univ. Press, Cambridge, 2004, pp.~1--20. \MR{2073343
  (2005i:20073)}

\bibitem{Antonyan}
Sergey~A. Antonyan, \emph{Proper actions on topological groups: applications to
  quotient spaces}, Proc. Amer. Math. Soc. \textbf{138} (2010), no.~10,
  3707--3716. \MR{2661569}

\bibitem{Arhangel}
Alexander Arhangel{\cprime}skii and Mikhail Tkachenko, \emph{Topological groups
  and related structures}, Atlantis Studies in Mathematics, vol.~1, Atlantis
  Press, Paris, 2008. \MR{2433295 (2010i:22001)}

\bibitem{Bell_and_Dranish}
G.~Bell and A.~Dranishnikov, \emph{Asymptotic dimension}, Topology Appl.
  \textbf{155} (2008), no.~12, 1265--1296. \MR{2423966 (2009d:55001)}

\bibitem{Carlsson_Goldfarb}
Gunnar Carlsson and Boris Goldfarb, \emph{On homological coherence of discrete
  groups}, J. Algebra \textbf{276} (2004), no.~2, 502--514. \MR{2058455
  (2005a:20078)}

\bibitem{Dranish_and_Smith}
A.~Dranishnikov and J.~Smith, \emph{Asymptotic dimension of discrete groups},
  Fund. Math. \textbf{189} (2006), no.~1, 27--34. \MR{2213160 (2007h:20041)}

\bibitem{Engelking}
Ryszard Engelking, \emph{General topology}, second ed., Sigma Series in Pure
  Mathematics, vol.~6, Heldermann Verlag, Berlin, 1989, Translated from the
  Polish by the author. \MR{1039321 (91c:54001)}

\bibitem{Grave}
Bernd Grave, \emph{Asymptotic dimension of coarse spaces}, New York J. Math.
  \textbf{12} (2006), 249--256 (electronic). \MR{2259239 (2007f:51023)}

\bibitem{Ji}
Lizhen Ji, \emph{Asymptotic dimension and the integral {$K$}-theoretic
  {N}ovikov conjecture for arithmetic groups}, J. Differential Geom.
  \textbf{68} (2004), no.~3, 535--544. \MR{2144540 (2006c:57025)}

\bibitem{Lueck}
Wolfgang L{\"u}ck, \emph{Survey on classifying spaces for families of
  subgroups}, Infinite groups: geometric, combinatorial and dynamical aspects,
  Progr. Math., vol. 248, Birkh\"auser, Basel, 2005, pp.~269--322. \MR{2195456
  (2006m:55036)}

\bibitem{Markov}
Andre\u{i}~Andreevi\u{c} Markov, \emph{On free topological groups}, Dokl. Akad.
  Nauk SSSR \textbf{31} (1941), 299--301.

\bibitem{Morris}
Sidney~A. Morris and H.~B. Thompson, \emph{Free topological groups with no
  small subgroups}, Proc. Amer. Math. Soc. \textbf{46} (1974), 431--437.
  \MR{0352318 (50 \#4805)}

\bibitem{Roe}
John Roe, \emph{Lectures on coarse geometry}, University Lecture Series,
  vol.~31, American Mathematical Society, Providence, RI, 2003. \MR{2007488
  (2004g:53050)}

\bibitem{Yu}
Guoliang Yu, \emph{The {N}ovikov conjecture for groups with finite asymptotic
  dimension}, Ann. of Math. (2) \textbf{147} (1998), no.~2, 325--355.
  \MR{1626745 (99k:57072)}

\end{thebibliography}

\def\cprime{$'$}
\providecommand{\bysame}{\leavevmode\hbox to3em{\hrulefill}\thinspace}
\providecommand{\MR}{\relax\ifhmode\unskip\space\fi MR }
\providecommand{\MRhref}[2]{%
  \href{http://www.ams.org/mathscinet-getitem?mr=#1}{#2}
}
\providecommand{\href}[2]{#2}

\end{document}